\documentclass[11pt,oneside,english,reqno]{nyjm}
\usepackage[T1]{fontenc}
\usepackage[latin9]{inputenc}
\usepackage{fancyhdr}
\usepackage{babel}
\usepackage{float}
\usepackage{mathrsfs}
\usepackage{mathtools}
\usepackage{amsthm}
\usepackage{amstext}
\usepackage{amssymb}
\usepackage{graphicx}
\usepackage{esint}
\usepackage[unicode=true,pdfusetitle,
 bookmarks=true,bookmarksnumbered=false,bookmarksopen=false,
 breaklinks=false,pdfborder={0 0 1},backref=false,colorlinks=false]
 {hyperref}
\hypersetup{
 colorlinks=true,citecolor=blue,linkcolor=blue,linktocpage=true}

\makeatletter
\numberwithin{equation}{section}
\numberwithin{figure}{section}
\numberwithin{table}{section}
\usepackage{enumitem}		
\theoremstyle{plain}
\newtheorem{thm}{\protect\theoremname}[section]
  \theoremstyle{definition}
  \newtheorem{defn}[thm]{\protect\definitionname}
  \theoremstyle{remark}
  \newtheorem{rem}[thm]{\protect\remarkname}
  \theoremstyle{plain}
  \newtheorem{lem}[thm]{\protect\lemmaname}
  \theoremstyle{plain}
  \newtheorem{cor}[thm]{\protect\corollaryname}
  \theoremstyle{definition}
  \newtheorem{example}[thm]{\protect\examplename}
  \theoremstyle{plain}
  \newtheorem{prop}[thm]{\protect\propositionname}
  \theoremstyle{remark}
  \newtheorem*{acknowledgement*}{\protect\acknowledgementname}

\providecommand{\MR}[1]{}

\AtBeginDocument{
  
}

\makeatother

  \providecommand{\acknowledgementname}{Acknowledgement}
  \providecommand{\corollaryname}{Corollary}
  \providecommand{\definitionname}{Definition}
  \providecommand{\examplename}{Example}
  \providecommand{\lemmaname}{Lemma}
  \providecommand{\propositionname}{Proposition}
  \providecommand{\remarkname}{Remark}
\providecommand{\theoremname}{Theorem}

\begin{document}
\thispagestyle{plain}

\lhead{}

\rhead{}

\chead{INDUCED REPRESENTATIONS ARISING FROM A CHARACTER}

\title{Induced representations arising from a character with finite orbit
in a semidirect product}

\author{Palle Jorgensen}

\address{(Palle Jorgensen) Department of Mathematics, The University of Iowa,
Iowa City, IA 52242-1419, U.S.A. }

\email{palle-jorgensen@uiowa.edu}

\urladdr{http://www.math.uiowa.edu/\textasciitilde{}jorgen/}

\author{Feng Tian}

\address{(Feng Tian) Department of Mathematics, Trine University, IN 46703,
U.S.A.}

\email{tianf@trine.edu}


\begin{abstract}
Making use of a unified approach to certain classes of induced representations,
we establish here a number of detailed spectral theoretic decomposition
results. They apply to specific problems from non-commutative harmonic
analysis, ergodic theory, and dynamical systems. Our analysis is in
the setting of semidirect products, discrete subgroups, and solenoids.
Our applications include analysis and ergodic theory of Bratteli diagrams
and their compact duals; of wavelet sets, and wavelet representations.

\tableofcontents{}
\end{abstract}

\maketitle

\section{Introduction}

The purpose of the present paper is to demonstrate how a certain induction
(from representation theory) may be applied to a number of problems
in dynamics, in spectral theory and harmonic analysis; yielding answers
in the form of explicit invariants, and equivalences. In more detail,
we show that a certain family of induced representations forms a unifying
framework. The setting is unitary representations, infinite-dimensional
semidirect products and crossed products. While the details in our
paper involve a variant of Mackey\textquoteright s theory, our approach
is extended, and it is constructive and algorithmic. Our starting
point is a family of imprimitivity systems. By this we mean a pair:
a unitary representation, and a positive operator-valued mapping,
subject to a covariance formula for the two (called imprimitivity).
One of its uses is a characterization of those unitary representations
of some locally compact group $G$ \textquotedblleft which arise\textquotedblright{}
as an induction from a subgroup; i.e., as induced from of a representation
of a specific subgroup of $G$. (Here, the notation \textquotedblleft which
arise\textquotedblright{} means \textquotedblleft up to unitary equivalence.\textquotedblright )

Below we give a summary of notation and terminologies, and we introduce
a family of induced representations, see e.g., \cite{Jor88,Mac88,Ors79}. 

The framework below is general; the context of locally compact (non-abelian)
groups. But we shall state the preliminary results in the context
of unimodular groups, although it is easy to modify the formulas and
the results given to non-unimodular groups. One only needs to incorporate
suitable factors on the respective modular functions, that of the
ambient group and that of the subgroup. Readers are referred to \cite{Ors79}
for additional details on this point. Another reason for our somewhat
restricted setting is that our main applications below will be to
the case of induction from suitable abelian subgroups. General terminology:
In the context of locally compact abelian groups, say $B$, we shall
refer to Pontryagin duality (see e.g., \cite{Rud62}); and so in particular,
when the abelian group $B$ is given, by \textquotedblleft the dual\textquotedblright{}
we mean the group of all continuous characters on $B$, i.e., the
one-dimensional unitary representations of $B$. We shall further
make use of Pontryagin's theorem to the effect that the double-dual
of $B$ is naturally isomorphic to $B$ itself.

Our results are motivated in part by a number of non-commutative harmonic
analysis issues involved in the analysis of wavelet representations,
and wavelet sets; see especially \cite{LPT01}; and also \cite{LSST06,MR2769882,MR3121682,MR3200927,MR3247020}.
\begin{defn}
Let $G$ be a locally compact group, $B\subset G$ a closed subgroup;
both assumed unimodular. Given  $V\in Rep\left(B,\mathscr{H}_{V}\right)$,
a representation of $B$ in the Hilbert space $\mathscr{H}_{V}$,
let 
\begin{equation}
U=Ind_{B}^{G}\left(V\right)\label{eq:id1}
\end{equation}
be the \emph{induced representation}. $U$ acts on the Hilbert space
$\mathscr{H}_{U}$ as follows:

$\mathscr{H}_{U}$ consists of all measurable functions $f:G\rightarrow\mathscr{H}_{V}$
s.t. 
\begin{align}
f\left(bg\right) & =V_{b}f\left(g\right),\;\forall b\in B,\:\forall g\in G;\;\mbox{and}\label{eq:id2}\\
\left\Vert f\right\Vert _{\mathscr{H}_{U}}^{2} & =\int_{B\backslash G}\left\Vert f\left(g\right)\right\Vert ^{2}d\mu_{B\backslash G}\left(g\right)<\infty\label{eq:id3}
\end{align}
w.r.t. the measure on the homogeneous space $B\backslash G$, i.e.,
the Hilbert space carrying the induced representation. Set 
\begin{equation}
\left(U_{g}f\right)\left(x\right):=f\left(xg\right),\quad x,g\in G.\label{eq:id4}
\end{equation}
\end{defn}
\begin{rem}
If the respective groups $B$ and $G$ are non-unimodular, we select
respective right-invariant Haar measures $db$, $dg$; and corresponding
modular functions $\triangle_{B}$ and $\triangle_{G}$. In this case,
the modification to the above is that eq. (\ref{eq:id2}) will instead
be 
\begin{equation}
f\left(bg\right)=\left(\frac{\triangle_{B}\left(b\right)}{\triangle_{G}\left(b\right)}\right)^{1/2}V_{b}f\left(g\right);\;b\in B,\:g\in G;\label{eq:id4-1}
\end{equation}
thus a modification in the definition of $\mathscr{H}(ind_{B}^{G}\left(V\right))$. 
\end{rem}
We recall the following theorems of Mackey \cite{Mac88,Ors79}. 
\begin{thm}[Mackey]
 $U=Ind_{B}^{G}\left(V\right)$ is a unitary representation. 
\end{thm}

\begin{thm}[Imprimitivity Theorem]
\label{thm:s1-4} A representation $U\in Rep\left(G,\mathscr{H}\right)$
is induced iff $\exists$ a positive operator-valued mapping
\begin{equation}
\pi:C_{c}\left(B\backslash G\right)\longrightarrow\mathscr{B}\left(\mathscr{H}\right)\;\mbox{s.t.}\label{eq:id5}
\end{equation}
\begin{equation}
U_{g}\pi\left(\varphi\right)U_{g}^{*}=\pi\left(R_{g}\varphi\right)\label{eq:id6}
\end{equation}
$\forall g\in G$, $\forall\varphi\in C_{c}\left(B\backslash G\right)$;
where $\pi$ is non-degenerate, and $R_{g}$ in (\ref{eq:id6}) denotes
the right regular action.\end{thm}
\begin{proof}
See \cite{Ors79,Jor88}.\end{proof}
\begin{thm}
Given $U\in Rep\left(G,\mathscr{H}\right)$, the following are equivalent:
\begin{enumerate}
\item $\exists\pi$ s.t. (\ref{eq:id6}) holds.
\item $\exists V\in Rep\left(B,\mathscr{H}_{V}\right)$ s.t. $U\cong Ind_{B}^{G}\left(V\right)$.
\item $L_{G}\left(U,\pi\right)\cong L_{B}\left(V\right)$
\item If $V_{i}\in Rep\left(B,\mathscr{H}_{V_{i}}\right)$, $i=1,2$ and
$U_{i}=Ind_{B}^{G}\left(V_{i}\right)$, $i=1,2$, then 
\[
L_{G}\left(\left(U_{1},\pi_{1}\right);\left(U_{2},\pi_{2}\right)\right)\cong L_{B}\left(V_{1},V_{2}\right);
\]
i.e., all intertwining operators $V_{1}\rightarrow V_{2}$ lift to
the pair $\left(U_{i},\pi_{i}\right)$, $i=1,2$. 
\end{enumerate}

(Here, ``$\cong$'' denotes unitary equivalence.) Specifically,
\begin{align}
 & L_{G}\left(\left(U_{1},\pi_{1}\right),\left(U_{2},\pi_{2}\right)\right)\nonumber \\
= & \Big\{ W:\mathscr{H}_{U_{1}}\rightarrow\mathscr{H}_{U_{2}}\;\big|\;WU_{1}\left(g\right)=U_{2}\left(g\right)W,\;\forall g\in G,\;\mbox{and}\nonumber \\
 & \quad U\pi_{1}\left(\varphi\right)=\pi_{2}\left(\varphi\right)W\Big\};
\end{align}
and 
\begin{equation}
L_{B}\left(V_{1},V_{2}\right)=\left\{ w:\mathscr{H}_{V_{1}}\rightarrow\mathscr{H}_{V_{2}}\:\big|\:wV_{1}\left(b\right)=V_{2}\left(b\right)w,\;\forall b\in B.\right\} 
\end{equation}

\end{thm}
We also recall the following result:
\begin{lem}
Let $B$ be a lattice in $\mathbb{R}^{d}$, let $\alpha$ be an action
of $\mathbb{Z}$ on $B$ by automorphisms, and let $(K=\widehat{B},\widehat{\alpha})$
represent the dual action of $\mathbb{Z}$ on the compact abelian
group $K$. Fix a $d\times d$ matrix $A$ preserving the lattice
$B$ with $spec\left(A\right)\subset\left\{ \lambda\::\:\left|\lambda\right|>1\right\} $. 

If $\alpha=\alpha_{A}\in Aut\left(B\right)$, then $\widehat{\alpha}\in Aut\left(K\right)$
is ergodic, i.e., if $\nu$ is the normalized Haar measure on $K$
and if $E\subset K$ is measurable s.t. $\widehat{\alpha}E=E$, then
$\nu\left(E\right)\left(1-\nu\left(E\right)\right)=0$.\end{lem}
\begin{proof}
By Halmos-Rohlin's theorem (see \cite{BJ91}), we must show that if
$b\in B$ and $A^{n}b=b$ for some $n\in\mathbb{N}$, then $b=0$.
This then follows from the assumption on the spectrum of $A$.
\end{proof}

\section{\label{sec:ind}The representation $Ind_{B}^{G}\left(\chi\right)$}

We now specify our notation and from this point on will restrict our
study to the following setting:
\begin{enumerate}[label=(\roman{enumi})]
\item $B$ -- a discrete abelian group (written additively)
\item $\alpha\in Aut\left(B\right)$
\item $G_{\alpha}:=B\rtimes_{\alpha}\mathbb{Z}$
\item $K:=\widehat{B}$, the compact dual group
\item $\widehat{\alpha}\in Aut(\widehat{B})$, dual action
\item $L_{\alpha}:=K\rtimes_{\widehat{\alpha}}\mathbb{Z}$, the $C^{*}$-algebra
crossed product \cite[p.299]{BJ91}; also written as $C^{*}\left(K\right)\rtimes_{\widehat{\alpha}}\mathbb{Z}$.
\end{enumerate}

More generally, let $K$ be a compact Hausdorff space, and $\beta:K\rightarrow K$
a homeomorphism; then we study the $C^{*}$-crossed product $C^{*}\left(K\right)\rtimes_{\beta}\mathbb{Z}$
(see \cite{BJ91}).

We define the \emph{induced representation}
\begin{equation}
U^{\chi}:=Ind_{B}^{G}\left(\chi\right)\label{eq:ir1}
\end{equation}
where $\chi\in K$, i.e., a character on $B$, and 
\begin{equation}
G:=B\rtimes_{\alpha}\mathbb{Z}\label{eq:ir2}
\end{equation}
as a semi-direct product. Note that $U^{\chi}$ in (\ref{eq:ir1})
is induced from a one-dimensional representation. 

Below, $\left\{ \delta_{k}\right\} _{k\in\mathbb{Z}}$ denotes the
canonical basis in $l^{2}\left(\mathbb{Z}\right)$, i.e., 
\[
\delta_{k}=\left(\cdots,0,0,1,0,0,\cdots\right)
\]
with ``1'' at the $k^{th}$ place. 
\begin{lem}
The representation $Ind_{B}^{G}\left(\chi\right)$ is unitarily equivalent
to 
\begin{equation}
U^{\chi}:G\rightarrow\mathscr{B}\left(l^{2}\left(\mathbb{Z}\right)\right),\;\mbox{where}\label{eq:p0}
\end{equation}
\begin{equation}
U_{\left(j,b\right)}^{\chi}\delta_{k}=\chi\left(\alpha_{k-j}\left(b\right)\right)\delta_{k-j},\quad\left(j,b\right)\in G.\label{eq:p1}
\end{equation}
\end{lem}
\begin{proof}
First note that both groups $B$ and $G=B\rtimes_{\alpha}\mathbb{Z}$
are discrete and unimodular, and the respective Haar measures are
the counting measure. Since $B\backslash G\simeq\mathbb{Z}$, so $\mathscr{H}^{\chi}=L^{2}\left(B\backslash G\right)\simeq l^{2}\left(\mathbb{Z}\right)$.
Recall the multiplication rule in $G$: 
\begin{equation}
\left(k,c\right)\left(j,b\right)=\left(k+j,\alpha_{k}\left(b\right)+c\right)\label{eq:p2.5-1}
\end{equation}
where $j,k\in\mathbb{Z}$, and $b,c\in B$.

The representation $Ind_{B}^{G}\left(\chi\right)$ acts on functions
$f:G\rightarrow\mathbb{C}$ s.t.
\begin{equation}
f\left(\left(j,b\right)\right)=\chi\left(b\right)f\left(j,0\right)\label{eq:p6}
\end{equation}
and  
\begin{equation}
\left\Vert f\right\Vert _{\mathscr{H}^{\chi}}^{2}=\sum_{j\in\mathbb{Z}}\left|f\left(j,0\right)\right|^{2}\label{eq:2.7}
\end{equation}

Moreover, the mapping 
\begin{equation}
l^{2}\ni\xi\longmapsto W\xi\in\mathscr{H}^{\lambda}\label{eq:2.8}
\end{equation}
given by 
\begin{equation}
\left(W\xi\right)\left(j,b\right)=\chi\left(b\right)\xi_{j}\label{eq:2.9}
\end{equation}
is a unitary intertwining operator, i.e., 
\begin{equation}
Ind_{B}^{G}\left(\chi\right)W\xi=WU^{\chi}\xi\label{eq:2.10}
\end{equation}
where
\begin{equation}
\left(U_{\left(j,b\right)}^{\chi}\xi\right)_{k}=\chi\left(\alpha_{k}\left(b\right)\right)\xi_{k+j}\label{eq:2.11}
\end{equation}
$\forall j,k\in\mathbb{Z}$, $\forall b\in B$, $\forall\xi\in l^{2}\left(\mathbb{Z}\right)$.
Note that (\ref{eq:2.11}) is equivalent to (\ref{eq:p1}).

To verify (\ref{eq:2.10}), we have 
\begin{eqnarray*}
\left(Ind_{B}^{G}\left(\chi\right)_{\left(j,b\right)}W\xi\right)\left(k,c\right) & = & W\xi\left(\left(k,c\right)\left(j,b\right)\right)\\
 & \underset{\eqref{eq:p2.5-1}}{=} & W\xi\left(k+j,\alpha_{k}\left(b\right)+c\right)\\
 & \underset{\eqref{eq:p6}}{=} & \chi\left(\alpha_{k}\left(b\right)+c\right)W\xi\left(k+j,0\right)\\
 & \underset{\eqref{eq:2.9}}{=} & \chi\left(\alpha_{k}\left(b\right)+c\right)\xi_{k+j}
\end{eqnarray*}
and 
\begin{eqnarray*}
\left(WU_{\left(j,b\right)}^{\chi}\xi\right)\left(k,c\right) & \underset{\eqref{eq:2.9}}{=} & \chi\left(c\right)\left(U_{\left(j,b\right)}^{\chi}\xi\right)_{k}\\
 & \underset{\eqref{eq:2.11}}{=} & \chi\left(c\right)\chi\left(\alpha_{k}\left(b\right)\right)\xi_{k+j}\\
 & = & \chi\left(\alpha_{k}\left(b\right)+c\right)\xi_{k+j}.
\end{eqnarray*}
In the last step we use that $\chi\left(\alpha_{k}\left(b\right)+c\right)=\chi\left(\alpha_{k}\left(b\right)\right)\chi\left(c\right)$,
which is just the representation property of $\chi:B\rightarrow\mathbb{T}=\left\{ z\in\mathbb{C}\:|\:\left|z\right|=1\right\} $,
\begin{equation}
\chi\left(b_{1}+b_{2}\right)=\chi\left(b_{1}\right)\chi\left(b_{2}\right),\quad\forall b_{1},b_{2}\in B.\label{eq:2.12}
\end{equation}
\end{proof}
\begin{rem}
In summary, the representation $Ind_{B}^{G}\left(\chi\right)$ has
three equivalent forms:
\begin{enumerate}
\item On $l^{2}\left(\mathbb{Z}\right)$, 
\begin{equation}
\left(Ind_{B}^{G}\left(\chi\right)\xi\right)_{k}=\chi\left(\alpha_{k}\left(b\right)\right)\xi_{k+j},\quad\left(\xi_{k}\right)\in l^{2}\left(\mathbb{Z}\right);\label{eq:2.24}
\end{equation}

\item Or in the ONB $\left\{ \delta_{k}\:|\:k\in\mathbb{Z}\right\} $, 
\begin{equation}
Ind_{B}^{G}\left(\chi\right)_{\left(j,k\right)}\delta_{k}=\chi\left(\alpha_{k-j}\left(b\right)\right)\delta_{k-j},\quad k\in\mathbb{Z};\label{eq:2.25}
\end{equation}

\item On $\mathscr{H}^{\chi}$, consisting of functions $f:G\rightarrow\mathbb{C}$
s.t. 
\[
f\left(j,b\right)=\chi\left(b\right)f\left(j,0\right),\quad\left\Vert f\right\Vert _{\mathscr{H}^{\lambda}}^{2}=\sum_{j\in\mathbb{Z}}\left|f\left(j,0\right)\right|^{2}<\infty
\]
where 
\begin{align*}
\left(Ind_{B}^{G}\left(\chi\right)_{\left(j,b\right)}f\right)\left(k,c\right) & =f\left(\left(k,c\right)\left(j,b\right)\right)\\
 & =f\left(\left(k+j,\alpha_{k}\left(b\right)+c\right)\right)\\
 & =\chi\left(\alpha_{k}\left(b\right)+c\right)f\left(k+j,0\right);
\end{align*}
with $\left(j,b\right),\left(k,c\right)\in G=B\rtimes_{\alpha}\mathbb{Z}$,
i.e., $j,k\in\mathbb{Z}$, $b,c\in B$.
\end{enumerate}
\end{rem}

\section{\label{sec:irr}Irreducibility}

Let $G=B\rtimes_{\alpha}\mathbb{Z}$, and $Ind_{B}^{G}\left(\chi\right)$
be the induced representation. Our main results are summarized as
follows:
\begin{enumerate}
\item If $\chi\in K=\widehat{B}$ has infinite order, then $Ind_{B}^{G}\left(\chi\right)\in Rep_{irr}\left(G,l^{2}\left(\mathbb{Z}\right)\right)$,
i.e., $Ind_{B}^{G}\left(\chi\right)$ is irreducible. We have 
\begin{equation}
Ind_{B}^{G}\left(\chi\right)_{\left(j,b\right)}=D_{\chi}\left(b\right)T_{j}\label{eq:s1}
\end{equation}
where $D_{\chi}\left(b\right)$ is diagonal, and $T_{j}:l^{2}\left(\mathbb{Z}\right)\rightarrow l^{2}\left(\mathbb{Z}\right)$,
$\left(T_{j}\xi\right)_{k}=\xi_{k+j}$, for all $\xi\in l^{2}\left(\mathbb{Z}\right)$. 
\item Suppose $\chi$ has finite order $p$, i.e., $\exists p$ s.t. $\widehat{\alpha}^{p}\chi=\chi$,
but $\widehat{\alpha}^{k}\chi\neq\chi$, $\forall1\leq k<p$. Set
$U_{p}^{\chi}=Ind_{B}^{G}\left(\chi\right)$, then 
\begin{equation}
U_{p}^{\chi}\left(j,b\right)=D_{\chi}\left(b\right)P^{j}\label{eq:s2}
\end{equation}
where $T_{j}\delta_{k}=\delta_{k-j}$ and $P=$ the $p\times p$ permutation
matrix, see (\ref{eq:perm}).
\end{enumerate}
Details below. 
\begin{lem}
\label{lem:pcom}Set $\left(T_{j}\xi\right)_{k}=\xi_{k+j}$, for $k,j\in\mathbb{Z}$,
$\xi\in l^{2}$; equivalently, $T_{j}\delta_{k}=\delta_{k-1}$. Then 
\begin{enumerate}
\item The following identity holds: 
\begin{equation}
T_{n}U^{\chi}=U^{\widehat{\alpha}^{n}\chi}T_{n},\quad\forall n\in\mathbb{Z},\forall\chi\in K.\label{eq:p4}
\end{equation}

\item For $\omega=\left(\omega_{i}\right)\in l^{\infty}\left(\mathbb{Z}\right)$,
set 
\begin{align}
\left(\pi\left(\omega\right)\xi\right)_{k} & =\omega_{k}\xi_{k},\quad\xi\in l^{2}\left(\mathbb{Z}\right);\;\mbox{or}\nonumber \\
\pi\left(\omega\right)\delta_{k} & =\omega_{k}\delta_{k}.\label{eq:14}
\end{align}
Then 
\begin{equation}
U_{\left(j,b\right)}^{\chi}\pi\left(\omega\right)\big(U_{\left(j,b\right)}^{\chi}\big)^{*}=\pi\left(T_{j}\omega\right)\label{eq:15}
\end{equation}
$\forall\left(j,b\right)\in G$, $\forall\chi\in K$.
\item Restriction to the representation:
\begin{equation}
U^{\chi}\Big|_{B}=\sum_{j\in\mathbb{Z}}^{\oplus}\left(\widehat{\alpha}^{j}\chi\right)\label{eq:16}
\end{equation}
Recall $G:=B\rtimes_{\alpha}\mathbb{Z}$ so the subgroup $B$ corresponds
to $j=0$ in $\left\{ \left(j,b\right)\:|\:j\in\mathbb{Z},b\in B\right\} $.
Note that $\widehat{\alpha}^{j}\chi$ is a one-dimensional representation
of $B$.
\end{enumerate}
\end{lem}
\begin{proof}
One checks that 
\[
\left(T_{n}U_{\left(j,b\right)}^{\chi}\xi\right)_{k}=\left(U_{\left(j,b\right)}^{\chi}\xi\right)_{k+n}=\chi\left(\alpha_{k+n}\left(b\right)\right)\xi_{k+n+j},\;\mbox{and}
\]
\[
\left(U_{\left(j,b\right)}^{\widehat{\alpha}^{n}\chi}T_{n}\xi\right)_{k}=\left(\widehat{\alpha}^{n}\chi\right)\left(\alpha_{k}\left(b\right)\right)\left(T_{n}\xi\right)_{k+j}=\chi\left(\alpha_{k+n}\left(b\right)\right)\xi_{k+n+j}.
\]
The other assertions are immediate. \end{proof}
\begin{cor}
If $\widehat{\alpha}^{p}\chi=\chi$, then $T_{p}$ commutes with $U^{\chi}$.
In this case, 
\begin{equation}
U_{\left(j,b\right)}^{\chi}\delta_{k}=\widehat{\alpha}^{k-j}\left(\chi\right)\left(b\right)\delta_{k-j}\label{eq:p2}
\end{equation}
induces an action on $l^{2}\left(\mathbb{Z}/p\mathbb{Z}\right)$.
If $m$ is fixed, the same representation is repeated, where the same
action occurs in the sub-band 
\begin{equation}
l^{2}\left(\left\{ \delta_{j+mp}\:|\:0\leq j<p\right\} \right).\label{eq:p3}
\end{equation}
\end{cor}
\begin{proof}
Immediate from Lemma \ref{lem:pcom}.
\end{proof}
The $p$-dimensional representation in each band $\left\{ \delta_{j+mp}\:|\:0\leq j<p\right\} $
may be given in matrix form:
\begin{lem}
\label{lem:DP}Let $\chi\in K$, and assume the orbit of $\chi$ under
the action of $\widehat{\alpha}$ has finite order $p$, i.e., $\widehat{\alpha}^{p}\chi=\chi$
and $\widehat{\alpha}^{k}\chi\neq\chi$, $1\leq k<p$. For $b\in B$,
set 
\[
D_{\chi}\left(b\right)=\begin{pmatrix}\chi\left(b\right)\\
 & \chi\left(\alpha\left(b\right)\right) &  & \smash{\scalebox{3}{\textup{0}}}\\
 &  & \ddots\\
\smash{\scalebox{3}{\textup{0}}} &  &  & \chi\left(\alpha^{p-1}\left(b\right)\right)
\end{pmatrix}=diag\left(\chi\left(\alpha^{i}\left(b\right)\right)\right)_{0}^{p-1}
\]
a $p\times p$ diagonal matrix. If $\chi=\left(\chi_{1},\ldots,\chi_{n}\right)$,
set $D_{\chi}\left(b\right)=diag\left(\chi_{i}\left(b\right)\right)_{n\times n}$. 

For the permutation matrix $P$, we shall use its usual $p\times p$
matrix representation
\begin{equation}
P=\begin{pmatrix}0 & 1 & 0 & \cdots & 0\\
0 & 0 & 1 & \cdots & 0\\
\vdots & \vdots & 0 & \cdots & \vdots\\
\vdots & \vdots & \vdots & \cdots & \vdots\\
\vdots & \vdots & \vdots &  & 0\\
0 & 0 & 0 & \cdots & 1\\
1 & 0 & 0 & \cdots & 0
\end{pmatrix}_{p\times p}.\label{eq:perm}
\end{equation}
We have that
\begin{enumerate}
\item The following identity holds:
\begin{equation}
PD_{\chi}\left(b\right)=D_{\chi}\left(\alpha\left(b\right)\right)P\label{eq:2.18}
\end{equation}

\item For all $\left(j,b\right)\in G=B\rtimes_{\alpha}\mathbb{Z}$, let
\begin{equation}
U_{p}^{\left(\chi\right)}\left(j,b\right):=D_{\chi}\left(b\right)P^{j},\label{eq:2.19}
\end{equation}
then $U_{p}^{\chi}\in Rep\left(G,\mathbb{C}^{p}\right)$, i.e., 
\[
U_{p}^{\chi}\left(j,b\right)U_{p}^{\chi}\left(j',b'\right)=U_{p}^{\chi}\left(j+j',\alpha_{j}\left(b'\right)+b\right)
\]
for all $\left(j,b\right),\left(j',b'\right)$ in $G$. 
\end{enumerate}
\end{lem}
\begin{proof}
It suffices to verify (\ref{eq:2.18}), and the rest of the lemma
is straightforward. Set $\chi_{k}:=\chi\left(\alpha^{k}\left(b\right)\right)$,
$0\leq k<p$, and 
\begin{equation}
D_{\chi}\left(b\right)=diag\left(\chi_{0},\chi_{1},\cdots,\chi_{p-1}\right).\label{eq:2.21}
\end{equation}
The assertion in (\ref{eq:2.18}) follows from a direct calculation.
We illustrate this with $p=3$, see Example \ref{exa:p=00003D3} below.\end{proof}
\begin{example}
\label{exa:p=00003D3}For $p=3$, (\ref{eq:2.18}) reads:
\begin{align*}
PD_{\chi}\left(b\right) & =\begin{pmatrix}0 & 1 & 0\\
0 & 0 & 1\\
1 & 0 & 0
\end{pmatrix}\begin{pmatrix}\chi_{0} & 0 & 0\\
0 & \chi_{1} & 0\\
0 & 0 & \chi_{2}
\end{pmatrix}=\begin{pmatrix}0 & \chi_{1} & 0\\
0 & 0 & \chi_{2}\\
\chi_{0} & 0 & 0
\end{pmatrix}\\
D_{\chi}\left(\alpha\left(b\right)\right)P & =\begin{pmatrix}\chi_{1} & 0 & 0\\
0 & \chi_{2} & 0\\
0 & 0 & \chi_{0}
\end{pmatrix}\begin{pmatrix}0 & 1 & 0\\
0 & 0 & 1\\
1 & 0 & 0
\end{pmatrix}=\begin{pmatrix}0 & \chi_{1} & 0\\
0 & 0 & \chi_{2}\\
\chi_{0} & 0 & 0
\end{pmatrix}
\end{align*}
\end{example}
\begin{lem}
\label{lem:irrA}Let $\chi\in K$, assume the orbit of $\chi$ under
the action of $\widehat{\alpha}$ has finite order $p$. Let $U_{p}^{\chi}$
be the corresponding representation. Then $U_{p}^{\chi}\in Rep_{irr}\left(G,l^{2}\left(\mathbb{Z}/p\mathbb{Z}\right)\right)$,
i.e., $U_{p}^{\chi}$ is irreducible.\end{lem}
\begin{proof}
Recall $U_{p}^{\chi}\left(j,b\right)=D_{\chi}\left(b\right)P^{j}$,
$\forall\left(j,b\right)\in G$; see (\ref{eq:2.19}). 

Let $A:l^{2}\left(\mathbb{Z}_{p}\right)\rightarrow l^{2}\left(\mathbb{Z}_{p}\right)$
be in the commutant of $U_{p}^{\chi}$, so $A$ commutes with $P$.
It follows that $A$ is a Toeplitz matrix
\begin{equation}
A=\begin{pmatrix}A_{0} & A_{2} & \ddots & \ddots & A_{p-1}\\
A_{-1} & A_{0} & A_{2} & \ddots & \ddots\\
\ddots & A_{-1} & A_{0} & \ddots & \ddots\\
\ddots & \ddots & \ddots & \ddots & \ddots\\
A_{-p+1} & \ddots & \ddots & \ddots & A_{0}
\end{pmatrix}\label{eq:2.22}
\end{equation}
relative to the ONB $\left\{ \delta_{0},\delta_{1},\cdots,\delta_{p-1}\right\} $,
where $A_{i,j}:=\left\langle \delta_{i},A\delta_{j}\right\rangle _{l^{2}}$.
Note that 
\begin{eqnarray*}
A_{i,j} & = & \left\langle \delta_{i},A\delta_{j}\right\rangle =\left\langle \delta_{i},AP\delta_{j+1}\right\rangle =\left\langle \delta_{i},PA\delta_{j+1}\right\rangle \\
 & = & \left\langle P^{*}\delta_{i},A\delta_{j+1}\right\rangle =\left\langle \delta_{i+1},A\delta_{j+1}\right\rangle =A_{i+1,j+1}
\end{eqnarray*}
where $i+1$, $j+1$ are the additions in $\mathbb{Z}_{p}=\mathbb{Z}/p\mathbb{Z}$,
i.e., $+\mod\:p$. 

If $A$ also commutes with $D$, then 
\begin{equation}
A_{k}\left(\chi\left(b\right)-\chi\left(\alpha^{k}\left(b\right)\right)\right)=0\label{eq:2.23}
\end{equation}
for all $k=1,2,\cdots,p-1$, and all $b\in B$. But since $\chi,\widehat{\alpha}\chi,\cdots,\widehat{\alpha}^{p-1}\chi$
are distinct, $A_{k}=0$, $\forall k\neq0$, and so $A=A_{0}I_{p\times p}$,
$A_{0}\in\mathbb{C}$. \end{proof}
\begin{prop}
\label{prop:decomp}The orbit of $\chi\in K$ under the action of
$\widehat{\alpha}$ has finite order $p$ if and only if 
\[
Ind_{B}^{G}\left(\chi\right)=\sum^{\oplus}\left(p\mbox{-dimensional irreducibles}\right)
\]
\end{prop}
\begin{proof}
This follows from an application of the general theory; more specifically
from an application of Mackey's imprimitivity theorem, in the form
given to it in \cite{Ors79}; see also Theorem \ref{thm:s1-4} above.\end{proof}
\begin{lem}
$Ind_{B}^{G}\left(\chi\right)$ is irreducible iff the orbit of $\chi$
under the action of $\widehat{\alpha}$ has finite order $p$, i.e.,
iff the set $\left\{ \widehat{\alpha}^{k}\chi\:|\:k\in\mathbb{Z}\right\} $
consists of distinct points.\end{lem}
\begin{proof}
We showed that if $\chi$ has finite order $p$, then the translation
$T_{p}$ operator commutes with $U^{\chi}:=Ind_{B}^{G}\left(\chi\right)$,
where $T_{p}\delta_{k}=\delta_{k-p}$, $k\in\mathbb{Z}$. Hence assume
$\chi$ has infinite order. 

Consider the cases $U^{\chi}\left(j,b\right)$ of 
\begin{enumerate}
\item $b=0$. $U^{\chi}\left(j,0\right)=T_{j}$, $j\in\mathbb{Z}$
\item $j=0$. $U^{\chi}\left(0,b\right)=D_{\chi}\left(b\right)$, $\delta_{k}\longmapsto\chi\left(\alpha_{k}\left(b\right)\right)\delta_{k}$,
where $\chi\left(\alpha_{k}\left(b\right)\right)=\left(\widehat{\alpha}^{k}\chi\right)\left(b\right)$,
and $D_{\chi}=\left(\widehat{\alpha}^{k}\chi\right)=$ diagonal matrix. 
\end{enumerate}

Thus if $A\in\mathscr{B}\left(l^{2}\left(\mathbb{Z}\right)\right)$
is in the commutant of $U^{\chi}$, then
\[
\left(A\xi\right)_{k}=\sum_{j\in\mathbb{Z}}\eta_{k-j}\xi_{j},\quad\eta\in l^{\infty}.
\]
But 
\begin{equation}
\left(\chi\left(\alpha_{s}\left(b\right)\right)-\chi\left(\alpha_{k}\left(b\right)\right)\right)\eta_{k-s}=0\label{eq:2.31}
\end{equation}
$\forall b$, $\forall k,s$, so $\eta_{t}=0$ if $t\in\mathbb{Z}\backslash\left\{ 0\right\} $
and $A=\eta_{0}I$. 

Note from (\ref{eq:2.31}) that if $k-s\neq0$ then $\exists b$ s.t.
\[
\chi\left(\alpha_{s}\left(b\right)\right)-\chi\left(\alpha_{k}\left(b\right)\right)\neq0
\]
since $\widehat{\alpha}^{s}\chi\neq\widehat{\alpha}^{k}\chi$. (Compare
with (\ref{eq:2.23}) in Lemma \ref{lem:irrA}.)\end{proof}
\begin{thm}
\label{thm:irrep}Let $Ind_{B}^{G}\left(\chi\right)$ be the induced
representation in (\ref{eq:ir1})-(\ref{eq:ir2}).
\begin{enumerate}
\item $Ind_{B}^{G}\left(\chi\right)$ is irreducible iff $\chi$ has no
finite periods.
\item Suppose the orbit of $\chi$ under the action of $\widehat{\alpha}$
has finite order $p$, i.e., $\widehat{\alpha}^{p}\chi=\chi$, and
$\widehat{\alpha}^{k}\chi\neq\chi$, $1\leq k<p$. Then the commutant
is as follows:
\begin{equation}
M_{p}:=\left\{ Ind_{B}^{G}\left(\chi\right)\right\} '\cong\left\{ f\left(z^{p}\right)\;\big|\;f\in L^{\infty}\left(\mathbb{T}\right)\right\} \label{eq:m1}
\end{equation}
where $\cong$ in (\ref{eq:m1}) denotes unitary equivalence.
\end{enumerate}
\end{thm}

\begin{thm}
Let $G=B\rtimes_{\alpha}\mathbb{Z}$, $K=\widehat{B}$, $\widehat{\alpha}\in Aut\left(K\right)$.
Assume the orbit of $\chi$ under the action of $\widehat{\alpha}$
has finite order $p$, i.e., $\widehat{\alpha}^{p}\chi=\chi$, $\widehat{\alpha}^{k}\chi\neq\chi$,
for all $1\leq k<p$. Then the representation $Ind_{B}^{G}\left(\chi\right)$
has abelian commutant 
\[
M_{p}=\left\{ Ind_{B}^{G}\left(\chi\right)\right\} '
\]
and $M_{p}$ does not contain minimal projections.\end{thm}
\begin{proof}
Follows from the fact that $M_{p}$ is $L^{\infty}\left(\mbox{Lebesgue}\right)$
as a von Neumann algebra, and this implies the conclusion. 

Details: Recall that $l^{2}\left(\mathbb{Z}\right)\simeq L^{2}\left(\mathbb{T}\right)$,
where 
\begin{align}
l^{2}\ni\xi & \longmapsto f_{\xi}\left(z\right)=\sum_{j\in\mathbb{Z}}\xi_{j}z^{j}\in L^{2}\left(\mathbb{T}\right)\label{eq:2.16}\\
T\xi & \longmapsto z^{-j}f_{\xi}\left(z\right),\quad z\in\mathbb{T}.\label{eq:2.17}
\end{align}
Note that $Ind_{B}^{G}\left(\chi\right)$ may be realized on $l^{2}\left(\mathbb{Z}\right)$
or equivalently on $L^{2}\left(\mathbb{T}\right)$ via (\ref{eq:2.16}),
where 
\[
\left\Vert f_{\xi}\right\Vert _{L^{2}}^{2}=\int_{\mathbb{T}}\left|f_{\xi}\right|^{2}=\sum_{k\in\mathbb{Z}l}\left|\xi_{k}\right|^{2}.
\]
On $l^{2}\left(\mathbb{Z}\right)$, we have 
\[
Ind_{B}^{G}\left(\chi\right)_{\left(j,b\right)}=D_{\chi}\left(b\right)T_{j};
\]
See (\ref{eq:s1}), and Lemma \ref{lem:DP}. 

And on $L^{2}\left(\mathbb{T}\right)$, we have 
\[
Ind_{B}^{G}\left(\chi\right)_{\left(j,b\right)}=\widehat{D}_{\chi}\left(b\right)\widehat{T}_{j},\;\mbox{where}
\]
$\widehat{D}_{\chi}\left(b\right)$ denotes rotation on $\mathbb{T}$,
extended to the solenoid; and $\widehat{T}_{j}=$ multiplication by
$z^{-j}$ acting on $L^{2}\left(\mathbb{T}\right)$. 

By (\ref{eq:m1}), $M_{p}$ is abelian and has no minimal projections.
Note projections in $L^{\infty}$ are given by $P_{E}=$ multiplication
by $\chi_{E}\left(z^{p}\right)$, where $E$ is measurable in $\mathbb{T}$. 
\end{proof}

\section{Super-representations}

By \textquotedblleft super-representation\textquotedblright{} we will
refer here to a realization of noncommutative relations \textquotedblleft inside\textquotedblright{}
certain unitary representations of suitable groups acting in enlargement
Hilbert spaces; this is in the sense of \textquotedblleft dilation\textquotedblright{}
theory \cite{FrKa02}, but now in a wider context than is traditional.
Our present use of \textquotedblleft super-representations\textquotedblright{}
is closer to that of \cite{MR2154344,MR2391800,MR2563094,MR2830593}. 
\begin{defn}
\label{def:s4-1}Let $G$ be a locally compact group, $B$ a given
subgroup of $G$, and let $\mathscr{H}_{0}$ be a Hilbert space. Let
$U_{0}:G\rightarrow\mathscr{B}\left(\mathscr{H}_{0}\right)$ be a
positive definite operator-valued mapping, i.e., for all finite systems
$\left\{ c_{j}\right\} _{j=1}^{n}\subset\mathbb{C}$, $\left\{ g_{j}\right\} _{j=1}^{n}\subset G$,
we have 
\[
\sum_{j}\sum_{k}\overline{c}_{j}c_{k}U_{0}(g_{j}^{-1}g_{k})\geq0
\]
in the usual ordering of Hermitian operators. 

If there is a Hilbert space $\mathscr{H}$, an isometry $V:\mathscr{H}_{0}\rightarrow\mathscr{H}$,
and a unitary representation $U:G\rightarrow\left(\mbox{unitary operators on \ensuremath{\mathscr{H}}}\right)$
such that
\begin{enumerate}
\item $U$ is induced from a unitary representation of $B$, and 
\item $U_{0}\left(g\right)=V^{*}U\left(g\right)V$, $g\in G$;
\end{enumerate}

then we say that $U$ is a \emph{super-representation.}

\end{defn}
Let $B$ be discrete, abelian as before, $K=\widehat{B}$, $\alpha\in Aut\left(B\right)$,
$\widehat{\alpha}\in Aut\left(K\right)$, and $G=B\rtimes_{\alpha}\mathbb{Z}$.
Set 

$Rep\left(G\right)$ - unitary representations; and

$Rep_{irr}\left(G\right)$ - irreducible representations in $Rep\left(G\right)$

For $\chi\in K$, let $O\left(\chi\right)$ be the orbit of $\chi$,
i.e., 
\begin{equation}
O\left(\chi\right)=\left\{ \widehat{\alpha}^{j}\left(\chi\right)\:|\:j\in\mathbb{Z}\right\} \label{eq:18}
\end{equation}

Given $U\in Rep\left(G\right)$, let $Class\left(U\right)=$ the equivalent
class of all unitary representations, equivalent to $U$; i.e.,
\begin{equation}
\begin{split}Class\left(U\right) & =\big\{ V\in Rep\left(G\right)\:\big|\:V\simeq U\big\}\\
 & =\big\{ V\in Rep\left(G\right)\:|\:\exists W,\:\mbox{unitary}\:\mbox{s.t.}\\
 & \qquad WV_{g}=U_{g}W,\;g\in G\big\}
\end{split}
\label{eq:20}
\end{equation}

For $U_{1},U_{2}\in Rep\left(G\right)$, set 
\begin{align*}
L\big(U^{\left(1\right)},U^{\left(2\right)}\big)= & \Big\{ W:\mathscr{H}\left(U_{1}\right)\rightarrow\mathscr{H}\left(U_{2}\right)\:|\:W\;\mbox{bounded}\:\mbox{s.t.}\:\\
 & \quad WU_{g}^{\left(1\right)}=U_{g}^{\left(2\right)}W,\;g\in G\Big\}
\end{align*}

\begin{thm}
The mapping $\left\{ \mbox{set of all orbits }O\left(\chi\right)\right\} \longrightarrow Class\left(Rep\left(G\right)\right)$
\[
K\ni\chi\longmapsto U^{\chi}:=Ind_{B}^{G}\left(\chi\right)\in Rep\left(G\right)
\]
passes to 
\begin{equation}
O\left(\chi\right)\longmapsto Class\left(U^{\chi}\right).\label{eq:21}
\end{equation}
\end{thm}
\begin{proof}
(Sketch) By Lemma \ref{lem:pcom} eq. (\ref{eq:p4}), if $\chi\in K$,
$j\in\mathbb{Z}$, then $T_{j}U^{\chi}T_{j}^{*}=U^{\widehat{\alpha}^{j}\chi}$
, and so $Class\left(U^{\chi}\right)$ depends only on $O\left(\chi\right)$
and not on the chosen point in $O\left(\chi\right)$. 

Using (\ref{eq:14})-(\ref{eq:16}), we can show that $U^{\chi}$
is irreducible, but if $O\left(\chi\right)$ is finite then we must
pass to the quotient $\mathbb{Z}_{p}:=\mathbb{Z}/p\mathbb{Z}$ and
realize $U^{\chi}$ on $l^{2}\left(\mathbb{Z}_{p}\right)$, as a finite-dimensional
representation.\end{proof}
\begin{thm}
There is a natural isomorphism:
\begin{align*}
L_{G}\left(U^{\chi_{1}},U^{\chi_{2}}\right) & \cong L_{B}\left(\chi_{2},U^{\chi_{1}}\big|_{B}\right)\\
 & =\sum_{j\in\mathbb{Z}}L_{B}\left(\chi_{2}\:\big|\:\widehat{\alpha}^{j}\left(\chi_{1}\right)\right)=\#\left\{ j\:\big|\:\chi_{2}=\widehat{\alpha}^{j}\left(\chi_{1}\right)\right\} .
\end{align*}
\end{thm}
\begin{proof}
Follows from Fröbenius reciprocity. 
\end{proof}
Another application of Fröbenius reciprocity: 
\begin{thm}
Let $\chi\in K$, and assume the orbit of $\chi$ under the action
of $\widehat{\alpha}$ has finite order $p$, i.e., $\widehat{\alpha}^{p}\chi=\chi$,
and $\widehat{\alpha}^{k}\chi\neq\chi$ if $1\leq k<p$. Let $U_{p}^{\chi}\in Rep\left(G,l^{2}\left(\mathbb{Z}_{p}\right)\right)$.
(Recall that $U_{p}^{\chi}$ is irreducible, $\dim U_{p}^{\chi}=p$.)
Then 
\[
L_{G}\left(U_{p}^{\chi},Ind_{B}^{G}\left(\chi\right)\right)\simeq L_{B}\left(\chi,U_{p}^{\chi}\big|_{B}\right)
\]
and 
\begin{equation}
\dim L_{B}\left(\chi,U_{p}^{\chi}\big|_{B}\right)=1.\label{eq:dim}
\end{equation}
\end{thm}
\begin{proof}
Since $U_{p}^{\left(\chi\right)}\left(j,b\right):=D_{\chi}\left(b\right)P^{j}$
(see (\ref{eq:2.19})), where $P=$ the permutation matrix on $\mathbb{Z}_{p}$,
we get 
\[
U_{p}^{\chi}\big|_{B}=\sum_{0\leq k<p}^{\oplus}\widehat{\alpha}^{k}\chi;
\]
see (\ref{eq:16}). But notice that all the $p$ characters $\chi,\widehat{\alpha}\left(\chi\right),\cdots,\widehat{\alpha}^{p-1}\left(\chi\right)$
are distinct, so (\ref{eq:dim}) holds since $L_{B}\left(\chi,\widehat{\alpha}^{k}\left(\chi\right)\right)=0$
if $k\neq0\mod p$.
\end{proof}
We have also proved the following:
\begin{thm}
Let $\chi\in K$, and assume the orbit of $\chi$ under the action
of $\widehat{\alpha}$ has finite order $p$. Assume $G$ is compact.
Let $U_{p}^{\chi}\in Rep_{irr}\left(G,l^{2}\left(\mathbb{Z}_{p}\right)\right)$.
Then $U_{p}^{\chi}$ is contained in $Ind_{B}^{G}\left(\chi\right)$
precisely once, i.e., 
\[
\dim L_{G}\left(U_{p}^{\chi},Ind_{B}^{G}\left(\chi\right)\right)=1.
\]

\end{thm}
But this form of Fröbenius reciprocity only holds for certain groups
$G$, e.g., when $G$ is compact. Now for $G=B\rtimes_{\alpha}\mathbb{Z}$,
the formal Fröbenius reciprocity breaks down, and in fact: 
\begin{thm}
$L_{G}\left(U_{p}^{\chi},Ind_{B}^{G}\left(\chi\right)\right)=0$ if
$\chi\in K$ is an element of finite order $p$. \end{thm}
\begin{proof}
We sketch the details for $p=3$ to simplify notation. For $p=3$,
\[
P=\begin{pmatrix}0 & 1 & 0\\
0 & 0 & 1\\
1 & 0 & 0
\end{pmatrix},\quad D_{\chi}\left(b\right)=\begin{pmatrix}\chi\left(b\right) & 0 & 0\\
0 & \chi\left(\alpha\left(b\right)\right) & 0\\
0 & 0 & \chi\left(\alpha_{2}\left(b\right)\right)
\end{pmatrix}
\]
and 
\begin{equation}
U_{p}^{\chi}\left(j,b\right)=D_{\chi}\left(b\right)P^{j}\label{eq:10.3}
\end{equation}
while 
\[
\left(Ind_{B}^{G}\left(\chi\right)_{\left(j,b\right)}\xi\right)_{k}=\chi\left(\alpha_{k}\left(b\right)\right)\xi_{k+j}
\]
$\forall\left(j,b\right)\in G$, $\forall k\in\mathbb{Z}$, $\forall\xi\in l^{2}\left(\mathbb{Z}_{p}\right)$. 

Let $W\in L_{G}\left(U_{p}^{\chi},Ind_{B}^{G}\left(\chi\right)\right)$,
and $u_{0},u_{1},u_{2}$ be the canonical basis in $\mathscr{H}\left(U_{p}^{\chi}\right)=\mathbb{C}^{3}$,
where 
\[
U_{p}^{\chi}\left(j,b\right)u_{k}=\chi\left(\alpha_{k+2j}\left(b\right)\right)u_{k+2j}\mod3
\]
$\forall\left(j,b\right)\in G$, $k\in\left\{ 0,1,2\right\} \simeq\mathbb{Z}/3\mathbb{Z}$. 

Set $Wu_{k}=\xi^{\left(k\right)}=(\xi_{s}^{\left(k\right)})_{s\in\mathbb{Z}}\in l^{2}\left(\mathbb{Z}\right)$,
where $\Vert\xi^{\left(k\right)}\Vert^{2}=\sum_{s\in\mathbb{Z}}|\xi_{s}^{\left(k\right)}|^{2}<\infty$. 

It follows that 
\[
WU_{p}^{\chi}\left(j,b\right)u_{k}=Ind_{B}^{G}\left(\chi\right)_{\left(j,b\right)}Wu_{k}
\]
$\forall\left(j,b\right)\in G$, $k\in\left\{ 0,1,2\right\} $. Thus
\[
\chi\left(\alpha_{k+2j}\left(b\right)\right)\xi_{s}^{\left(k+2j\right)_{3}}=\chi\left(\alpha_{s}\left(b\right)\right)\xi_{s+j}^{\left(k\right)},\quad\forall s,j\in\mathbb{Z}.
\]
Now set $j=3t\in3\mathbb{Z}$, and we get 
\[
\chi\left(\alpha_{k}\left(b\right)\right)\xi_{s}^{\left(k\right)}=\chi\left(\alpha_{s}\left(b\right)\right)\xi_{s+3t}^{\left(k\right)}
\]
 and 
\begin{equation}
\left|\xi_{s}^{\left(k\right)}\right|=\left|\xi_{s+3t}^{\left(k\right)}\right|,\quad\forall s,t\in\mathbb{Z}.\label{eq:tail}
\end{equation}
Since $\xi^{\left(k\right)}\in l^{2}\left(\mathbb{Z}\right)$, $\lim_{t\rightarrow\infty}\xi_{s+3t}^{\left(k\right)}=0$.
We conclude from (\ref{eq:tail}) that $\xi^{\left(k\right)}=0$ in
$l^{2}\left(\mathbb{Z}\right)$. \end{proof}
\begin{rem}
The decomposition of $Ind_{B}^{G}\left(\chi\right):G\longrightarrow B\left(l^{2}\left(\mathbb{Z}\right)\right)$
is still a bit mysterious. Recall this representation commutes with
$T_{3}:\left(\xi_{k}\right)\longmapsto\left(\xi_{k+3}\right)$; or
equivalently via $f\left(z\right)\longmapsto z^{3}f\left(z\right)$.
So it is not irreducible. 

The reason for $L_{G}\left(U_{p}^{\chi},Ind_{B}^{G}\left(\chi\right)\right)=0$,
e.g., in the case of $p=3$, is really that there is no isometric
version of the $3\times3$ permutation matrix $P$ in $T$, where
$T\delta_{k}=\delta_{k-1}$. 

If $W:\mathbb{C}^{3}\longrightarrow l^{2}\left(\mathbb{Z}\right)$
is bounded, $WP=TW$, then applying the polar decomposition 
\[
W=\left(W^{*}W\right)^{1/2}V
\]
with $V:\mathbb{C}^{3}\rightarrow l^{2}$ isometric, and $VP=TV$,
or 
\[
P=V^{*}TV,\quad P^{j}=V^{*}T_{j}V.
\]
So $T_{j}$ has the form
\[
T_{j}=\begin{pmatrix}P^{j} & *\\
0 & *
\end{pmatrix}.
\]
Pick $u\in\mathbb{C}^{3}$; then $\left\langle u,T_{j}u\right\rangle =\left\langle u,P^{j}u\right\rangle $.
However, $\left\langle u,T_{j}u\right\rangle \rightarrow0$ by Riemann-Lebesgue;
while $\left\langle u,P^{j}u\right\rangle \nrightarrow0$ since $P^{3}=I$. \end{rem}
\begin{defn}
A group $L$ acts on a set $S$ if there is a mapping 
\[
L\times S\longrightarrow S,\quad\left(\lambda,s\right)\longmapsto\lambda\left[s\right],\quad\lambda\lambda'[s]=\lambda[\lambda'[s]]
\]
and $\lambda^{-1}$ is the inverse of $S\ni s\longmapsto\lambda\left[s\right]\in S$,
a bijection. Often $S$ will have the structure of a topological space
or will be equipped with a $\sigma$-algebra of measurable sets.\end{defn}
\begin{rem}
While we have stressed discrete decompositions of induced representations,
the traditional literature has stressed direct integral decompositions,
see e.g., \cite{Mac88}. A more recent use of continuous parameters
in decompositions is a construction by J. Packer et al \cite{LPT01}
where ``wavelet sets'' arise as sets of support for direct integral
measures. In more detail, starting with a wavelet representation of
a certain discrete wavelet group (an induced representation of a semidirect
product), the authors in \cite{LPT01} establish a direct integral
where the resulting measure is a subset of R called ``wavelet set.''
These wavelet sets had been studied earlier, but independently of
representation theory. In the dyadic case, a wavelet set is a subset
of $\mathbb{R}$ which tiles $\mathbb{R}$ itself by a combination
of $\mathbb{Z}$-translations, and scaling by powers of 2. In the
Packer et. al. case, $\mathbb{R}\hookrightarrow K$, and
\[
E\;\mbox{wavelet set}\Longleftrightarrow\int_{E}^{\oplus}Ind_{B}^{G}\left(\chi_{t}\right)dt\;\mbox{is the wavelet representation in}\:L^{2}\left(\mathbb{R}\right).
\]
\end{rem}
\begin{thm}
The group $L:=K\rtimes_{\widehat{\alpha}}\mathbb{Z}$ acts on the
set $Rep\left(G\right)$ by the following assignment: 

Arrange it so that all the representations $U^{\chi}:=Ind_{B}^{G}\left(\chi\right)$,
$\chi\in K$, acts on the same $l^{2}$-space. Then $Rep\left(G\right)\sim Rep\left(G,l^{2}\right)$. 
\end{thm}

\section{Induction and Bratteli diagrams}

A Bratteli diagram is a group $G$ with vertices $V$, and edges $E\subset V\times V\backslash\left\{ \mbox{diagonal}\right\} $.
It is assumed that 
\[
V=\bigcup_{n=0}^{\infty}V_{n},
\]
as a disjoint union in such a way that the edges $E$ in $G$ can
be arranged in a sequnce of lines $V_{n}\rightarrow V_{n+1}$, so
no edge links pairs of vertices at the same level $V_{n}$. With a
system of inductions and restrictions one then creates these diagrams. 

Let $G$ be created as follows: For a given $V_{n}$, let the vertices
in $V_{n}$ represent irreducible representations of some group $G_{n}$,
and assume $G_{n}$ is a subgroup in a bigger group $G_{n+1}$. For
the vertices in $V_{n+1}$, i.e., in the next level in the Bratteli
diagram $G$, we take the irreducible representations occurring in
the decomposition of each of the restrictions:
\[
Ind_{G_{n}}^{G_{n+1}}\left(L\right)\Big|_{G_{n}},
\]
so the decomposition of the restriction of the induced representation;
see Figure \ref{fig:s5-1}.

Multiple lines in a Bratteli diagram count the occurrence of irreducible
representations with multiplicity; these are called \emph{multiplicity
lines}. 

Counting multiplicity lines at each level $V_{n}\rightarrow V_{n+1}$
we get a so called \emph{incidence matrix}. For more details on the
use of Bratteli diagrams in representation theory, see \cite{MR1804950,MR1869063,MR1889566,MR2030387,MR3150704,MR3266986,150801253B}.

\begin{figure}[H]
\includegraphics[width=0.6\textwidth]{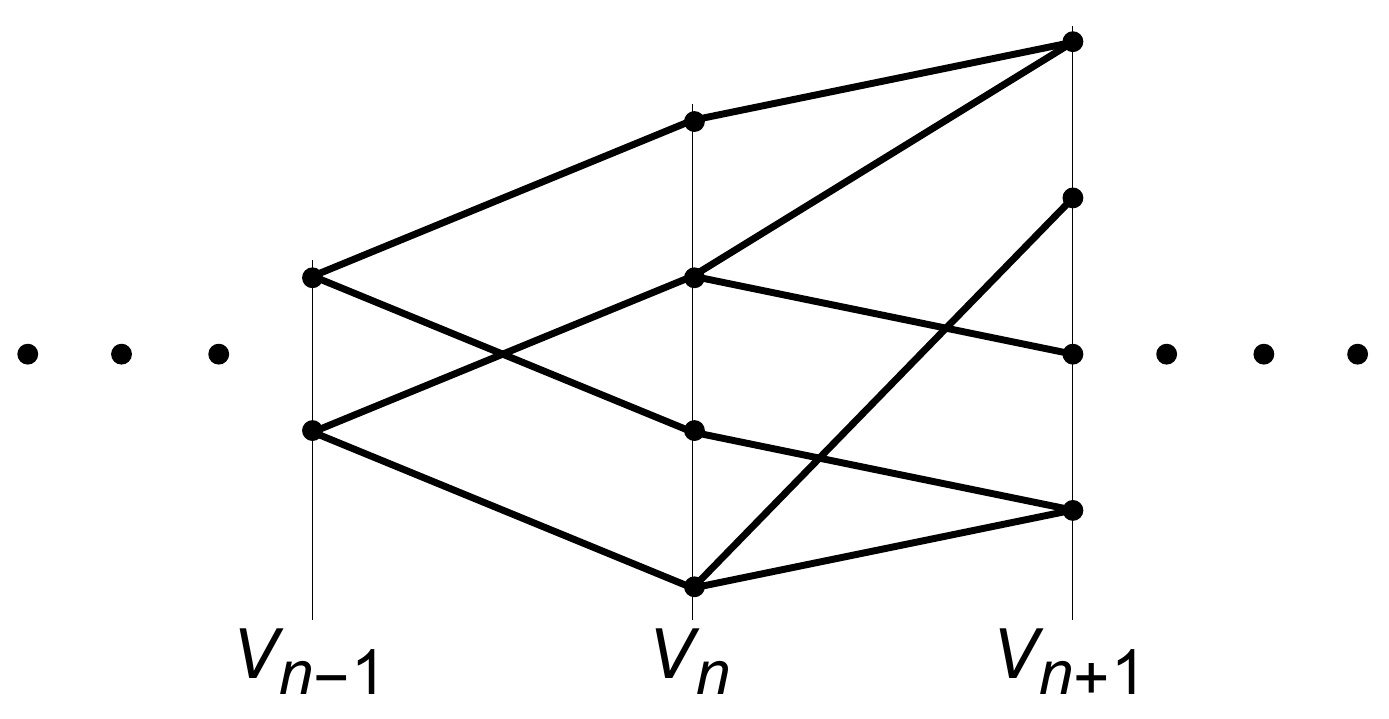}

\protect\caption{\label{fig:s5-1}Part of a Bratteli diagram.}
\end{figure}

While there is a host of examples from harmonic analysis and dynamical
systems; see the discussion above, we shall concentrate below on a
certain family of examples when the discrete group $B$ in the construction
(from Sections \ref{sec:ind} and \ref{sec:irr}) arises as a discrete
Bratteli diagram, and, as a result, its Pontryagin dual compact group
$K$ is a solenoid; also often called a compact Bratteli diagram. 

Example \ref{exa:A} below illustrates this in the special case of
constant incidence diagrams, but much of this discussion applies more
generally. 
\begin{example}
Let $B=\mathbb{Z}\left[\frac{1}{2}\right]$, dyadic rationals in $\mathbb{R}$.
Let $\alpha_{2}=$ multiplication by $2$, so that $\alpha_{2}\in Aut\left(B\right)$,
and set $G_{2}:=B\rtimes_{\alpha_{2}}\mathbb{Z}$. 

Recall the Baumslag-Solitar group $G_{2}$ with two generators $\left\{ u,t\right\} $,
satisfying $utu^{-1}=t^{2}$. With the correspondence, 
\[
u\longleftrightarrow\begin{pmatrix}2 & 0\\
0 & 1
\end{pmatrix},\quad t\longleftrightarrow\begin{pmatrix}1 & 1\\
0 & 1
\end{pmatrix}
\]
$\alpha_{2}$ acts by conjugation, 
\[
\begin{pmatrix}2 & 0\\
0 & 1
\end{pmatrix}\begin{pmatrix}1 & 1\\
0 & 1
\end{pmatrix}\begin{pmatrix}2^{-1} & 0\\
0 & 1
\end{pmatrix}=\begin{pmatrix}1 & 2\\
0 & 1
\end{pmatrix}.
\]
In particular, $T^{k}=\begin{pmatrix}1 & k\\
0 & 1
\end{pmatrix}$, $k=1,2,3,\cdots$ . 

We have 
\[
Ind_{\mathbb{Z}\left[\frac{1}{2}\right]}^{G_{2}}\left(\chi\right)\begin{pmatrix}2^{j} & b\\
0 & 1
\end{pmatrix}=D_{\chi}\left(b\right)T^{j}.
\]
Moreover, $\left\{ 2^{j}\::\:j\in\mathbb{Z}\right\} \simeq\mathbb{Z}\left[\frac{1}{2}\right]\backslash G_{2}\simeq\mathbb{Z}$.
Set $\left(T\xi\right)_{k}:=\xi_{k+1}$, $\xi\in l^{2}\left(\mathbb{Z}\right)$,
and let 
\[
D_{\chi}\left(b\right)=\begin{pmatrix}\ddots\\
 & \chi\left(b\right) &  &  & \smash{\scalebox{3}{0}}\\
 &  & \chi\left(2b\right)\\
 &  &  & \chi\left(3b\right)\\
\smash{\scalebox{3}{0}} &  &  &  & \ddots
\end{pmatrix},\quad\forall b\in\mathbb{Z}\left[\tfrac{1}{2}\right]
\]
Note that 
\[
TD_{\chi}\left(b\right)=D_{\chi}\left(\alpha_{2}\left(b\right)\right)T=D_{\widehat{\alpha}_{2}\chi}\left(b\right)T,\quad\forall b\in\mathbb{Z}\left[\tfrac{1}{2}\right].
\]

\end{example}

\begin{example}
\label{exa:A}Let $A$ be a $d\times d$ matrix over $\mathbb{Z}$,
$\det A\neq0$; then 
\[
\mathbb{Z}^{d}\hookrightarrow A^{-1}\mathbb{Z}^{d}\hookrightarrow A^{-2}\mathbb{Z}^{d}\hookrightarrow\cdots
\]
and set 
\[
B_{A}:=\bigcup_{k\geq0}A^{-k}\mathbb{Z}^{d}.
\]
Let $\alpha_{A}=$ multiplication by $A$, so that $\alpha_{A}\in Aut\left(B_{A}\right)$,
and $G_{A}:=B_{A}\rtimes_{\alpha_{A}}\mathbb{Z}$. 

The corresponding standard wavelet representation $U_{A}$ will now
act on the Hilbert space $L^{2}\left(\mathbb{R}^{d}\right)$ with
$d$-dimensional Lebesgue measure; and $U_{A}$ is a unitary representation
of the discrete matrix group $G_{A}=\left\{ \left(j,\beta\right)\right\} _{j\in\mathbb{Z},\beta\in B_{A}}$,
specified by
\begin{equation}
\left(j,\beta\right)\left(k,\gamma\right)=\left(j+k,\beta+A^{j}\gamma\right),\;\mbox{and}\label{eq:w1}
\end{equation}
defined for all $j,k\in\mathbb{Z}$, and $\beta,\gamma\in B_{A}$.
Alternatively, the group from (\ref{eq:w1}) may be viewed in matrix
form as follows:
\[
\left(j,\beta\right)\longrightarrow\begin{pmatrix}A^{j} & \beta\\
0 & 1
\end{pmatrix},\quad\mbox{and }\;\left(k,\gamma\right)\longrightarrow\begin{pmatrix}A^{k} & \gamma\\
0 & 1
\end{pmatrix}.
\]

The wavelet representation $U_{A}$ of $G_{A}$ acting on $L^{2}\left(\mathbb{R}^{d}\right)$
is now 
\begin{equation}
U_{A}\left(A\right)\left(x\right):=\left(\det A\right)^{-\frac{j}{2}}f\left(A^{-j}\left(x-\beta\right)\right),
\end{equation}
defined for all $f\in L^{2}\left(\mathbb{R}^{d}\right)$, $\forall j\in\mathbb{Z}$,
$\forall\beta\in B_{A}$, and $x\in\mathbb{R}^{d}$.\end{example}
\begin{acknowledgement*}
The co-authors thank the following colleagues for helpful and enlightening
discussions: Professors Daniel Alpay, Sergii Bezuglyi, Ilwoo Cho,
Ka Sing Lau, Azita Mayeli, Paul Muhly, Myung-Sin Song, Wayne Polyzou,
Keri Kornelson, and members in the Math Physics seminar at the University
of Iowa. 

The authors are very grateful to an anonymous referee for a list of
corrections and constructive suggestions. The revised paper is now
much better, both in substance, and in its presentation.
\end{acknowledgement*}
\bibliographystyle{amsalpha}
\bibliography{ref}

\end{document}